\documentclass[12pt,reqno]{amsart}
\usepackage{amsmath}
\usepackage{}
\usepackage{}
\usepackage{amsfonts}
\usepackage{a4wide}
\usepackage{amsmath,amsthm,amssymb,amscd}
\usepackage{latexsym}
\usepackage{hyperref}
\usepackage[numbers,sort&compress]{natbib}
\usepackage{hypernat}
\allowdisplaybreaks
\numberwithin{equation}{section}

\newtheorem{theorem}{Theorem}[section]
\newtheorem{proposition}[theorem]{Proposition}
\newtheorem{corollary}[theorem]{Corollary}
\newtheorem{lemma}[theorem]{Lemma}

\theoremstyle{definition}

\newtheorem{remark}[theorem]{Remark}

\newcommand{\ds}{\displaystyle}

\begin{document}

\title[Fractional Schr\"{o}dinger Equations]
{Semi-classical solutions for Fractional Schr\"{o}dinger Equations with Potential Vanishing at Infinity}

 \author{Xiaoming An,\,\,\,\,Shuangjie Peng\,\,\,and\,\,\,Chaodong Xie }

 \address{School of Mathematics and Statistics\, \&\, Hubei Key Laboratory of Mathematical Sciences, Central China Normal
University, Wuhan, 430079, P. R. China}

\email{ 651547603@qq.com}

\address{School of Mathematics and Statistics\, \&\, Hubei Key Laboratory of Mathematical Sciences, Central China
Normal University, Wuhan, 430079, P. R. China }

\email{ sjpeng@mail.ccnu.edu.cn}

\address{School of Economics\,\& Management, Guizhou University for Ethinic Minorities, Guiyang, 550025, P. R. China}

\email{ 776527080@qq.com}

\begin{abstract}
We study the following fractional Schr\"{o}dinger equation
\begin{equation}\label{eq0.1}
\varepsilon^{2s}(-\Delta)^s u + Vu = |u|^{p - 2}u, \,\,x\in\,\,\mathbb{R}^N.
\end{equation}
We show that if the external potential $V\in C(\mathbb{R}^N;[0,\infty))$ has a local minimum and $p\in (2 + 2s/(N - 2s), 2^*_s)$, where $2^*_s=2N/(N-2s),\,N\ge 2s$, the problem has a family of solutions concentrating at the local minimum of $V$ provided that $\liminf_{|x|\to \infty}V(x)|x|^{2s} > 0$. The proof is based on  variational methods and  penalized technique.

\end{abstract}

\maketitle

\section{Introduction and main results}

In this paper, we consider the fractional Schr\"{o}dinger equation
\begin{equation}\label{eq1.1}
\varepsilon^{2s}(-\Delta)^s u + Vu = |u|^{p - 2}u,\,\,x\in\,\,\mathbb{R}^N,
\end{equation}
where $N > 2s$, $s\in (0,1)$, $V$ is  continuous function, $\varepsilon > 0$ is a small parameter, $p\in (2 + 2s/(N - 2s), 2^*_s)$, $2^*_s=2N/(N-2s)$. Problem~\eqref{eq1.1}  is derived from the study of time-independent waves $\psi(x,t) = e^{-iEt}u(x)$ of the following nonlinear fractional Schr\"{o}dinger equation
$$
i\varepsilon\frac{\partial{\psi}}{{\partial t}} = \varepsilon^{2s}(-\Delta)^s{\psi} + U(x)\psi - f(\psi)\ x\in \mathbb{R}^N.
\eqno(NLFS)
$$
Indeed, letting $f(t) = |t|^{p - 2}t$ and inserting $\psi(x,t) = e^{-iEt}u(x)$ into $(NLFS)$, one can obtain \eqref{eq1.1} with  $V(x) = U(x) - E$.

When $s = \frac{1}{2}$, equation \eqref{eq1.1} can be used to describe some properties of Einstein's theory of relativity and also has been derived as models of many physical phenomena, such as phase transition, conservation laws, especially in fractional quantum mechanics, etc., \cite{1}. $(NLFS)$ was introduced by Laskin (\cite{2}, \cite{3}) as an extension of the classical nonlinear Schr\"{o}dinger equations s = 1 in which the Brownian motion of the quantum
paths is replaced by a L\`{e}vy flight. To see more physical backgrounds, we refer to \cite{4}.

When $s = 1$, the nonlinear Schr\"{o}dinger equation
$$
-\varepsilon^2\Delta u + Vu = |u|^{p - 2}u,
\eqno(NLS)
$$
has been extensively explored in the semiclasical scope and a valuable amount of work has been done, exhibiting that existence and concentration phenomena of single-peak and multiple-peak solutions occur at critical points or minimums of the electric potential $V$ when $\varepsilon\to 0$, see e.g.
\cite{{5},{6},{7},{8},{9},{10},{11},{12},{13}}. The concentrating phenomena occurs at higher dimensional sets has been  handled by  Ambrosetti et al. in \cite{14}. Then Del Pino, Kowalczyk and Wei in \cite{15} have completely settled the conjecture proposed by Ambrosetti et al in \cite{14}, which says  that, the concentration may  happen at a closed hypersurface. Recently, the existence of concentration driven by an external magnetic field has been invested by Bonheure et al. in \cite{16}. When $\varepsilon$ is fixed and $V$ is positive, the existence of infinitely many positive solutions to equation $(NLS)$  has also been settled down, see Wei and Yan's work \cite{17} under the assumption that  $V$ is radial symmetric, and see the results of Cerami et al \cite{18} if  $V$ does not satisfy any symmetric condition.

When $0<s<1$, let $\varepsilon = 1$ and $V \equiv \lambda > 0$, equation \eqref{eq1.1} is
\begin{equation}\label{eq1.2}
  (-\Delta)^s u + \lambda u = |u|^{p - 2}u, \,\,x\in\,\,\mathbb{R}^N.
\end{equation}
The existence of the least energy solution to \eqref{eq1.2} has been proved in [21, 22] by using the concentration-compactness arguments. The symmetry and optimal decay estimate for positive solutions of \eqref{eq1.2} have been showed by using the moving-plane method in \cite{23}.
The essential observation about nondegeneracy of ground states was obtained by  Frank and  Lenzmann in \cite{19} for the case $N=1$, and  Frank,  Lenzmann and Silvester  \cite{20} for  high dimensional case.

For the nonlocal case $0<s<1$, to our best knowledge, there are few results on the concentration phenomena of solutions of \eqref{eq1.1} in the semiclassical case, i.e., $\varepsilon \to 0$.
Concerning the non-vanishing case, i.e., $\inf_{x\in \mathbb{R}^N}V(x) > 0$, Alves et al. used the penalized method developed by Del Pino et al  in \cite{10} and extension method developed by Caffarelli et al. in \cite{24} to construct a family of positive solutions for \eqref{eq1.1}, which concentrate at the global minimum of $V$ when $\varepsilon\to 0$, to see more details, we refer to \cite{100}. We also refer to \cite{{25},{26},{27}}, which considered the concentration phenomena occuring at the nondegenerate critical points of $V$ by using a reduction method which is based on the nondegeneracy of the  ground states proved by  Frank and  Lenzmann in \cite{19} and Frank,  Lenzmann and Silvester  in \cite{20}.

In this paper, we consider the existence and asymptotic behavior of  semiclassical solutions to the nonlocal Schr\"{o}dinger equations~\eqref{eq1.1}  with vanishing potential.  As far as we know, for \eqref{eq1.1} with  vanishing potential, it seems that  there are no concentration results. The main difficulties lie in  the following two aspects: firstly, comparing with the case $s=1$,  we need the global $L^2$-norm estimate on the solutions when we use the penalized argument,  we can  see Proposition \ref{pr3.3} later; secondly, the vanishing of $V$ needs us  to construct penalized functions more explicit than the non-vanishing case. This step is much harder than the local case $s = 1$. In fact, as one can see in Section 4, the nonlocal effect makes the estimate on the fractional Laplacian $(-\Delta)^sf$ more tricky than that of  $-\Delta f$.

Here, we should  mention the penalized idea proposed by Schaftingen et al. in \cite{28}, dealing with the vanishing $V$ too. They handled with the following classical Schr\"{o}dinger equation with Choquard term
$$
-\varepsilon^2\Delta u + Vu = \varepsilon^{\alpha}(I_{\alpha}\ast |u|^p)|u|^{p - 2}u,\,\,x\in \mathbb{R}^N.
$$
The  idea there of penalizing  the nonlinearity   inspired us to deal with \eqref{eq1.1}. However, as we can see later, we should construct a different penalized functional  for the  fractional Laplacian case.

For  $s\in(0,1)$, the fractional Sobolev space $H^s(\mathbb{R}^N)$ is defined by
$$
H^s(\mathbb{R}^N) = \Big\{u\in L^2(\mathbb{R}^N):\frac{u(x) - u(y)}{|x - y|^{N/2 + s}}\in L^2(\mathbb{R}^N\times\mathbb{R}^N)\Big\},
$$
endowed with the norm
$$
\|u\|_{H^s(\mathbb{R}^N)} = \Big(\int_{\mathbb{R}^N}|(-\Delta)^{s/2}u|^2 + u^2dx \Big)^{\frac{1}{2}},
$$
where
$$
\int_{\mathbb{R}^N}|(-\Delta)^{s/2}u|^2dx = \int_{\mathbb{R}^{2N}}\frac{|u(x) - u(y)|^2}{|x - y|^{N + 2s}}dxdy.
$$
Here we also  give the definition  of the fractional Laplacian which will be used later (see \cite{4} for example):
$$
(-\Delta)^s u = \int_{\mathbb{R}^N}\frac{u(x) - u(y)}{|x - y|^{N + 2s}}dy\ \ \ \text{for all}\ u\in H^s(\mathbb{R}^N).
$$
Our work is based on the following weighted Hilbert space:
$$
H^s_{V,\varepsilon}(\mathbb{R}^N) = \left\{(-\Delta)^{s/2}u\in L^2(\mathbb{R}^N):\,\,u\in L^2(\mathbb{R}^N,V(x)dx)\right\},
$$
endowed with the norm
$$
\|u\|_{H^s_{V,\varepsilon}(\mathbb{R}^N)} = \Big(\int_{\mathbb{R}^N}\varepsilon^{2s}|(-\Delta)^{s/2}u|^2 + Vu^2dx \Big)^{\frac{1}{2}}.
$$

In the sequel, we set
$$
2^*_s = \left\{
          \begin{array}{ll}
            \frac{2N}{N - 2s}, & \text{for}\ N > 2s, \vspace{2mm}\\
            +\infty, & \text{for}\ N \leq 2s.
          \end{array}
        \right.
$$
We assume that $V\in C(\mathbb{R}^N,[0,\infty))$ satisfies the following assumptions

$(V_1)$  There exists an open bounded set $\Lambda$ with  smooth boundary $\partial\Lambda$,  such that
\begin{equation*}
\inf_{\Lambda}V < \inf_{\partial\Lambda}V.
\end{equation*}
Without loss of generality, we assume that $0\in \Lambda$.

$(V_2)$ $\liminf_{|x|\to\infty}V(x)|x|^{2s} > 0$.
\vspace{1cm}

\noindent Our main result can be stated as follows.

\begin{theorem}\label{th1.1}
Let $N > 2s$, $s\in(0,1)$,  $p\in (2 + 2s/(N - 2s), 2^*_s)$ and $V\in C(\mathbb{R}^N;[0,\infty))$ satisfy $(V_1)$ and $(V_2)$. Then problem \eqref{eq1.1} has a  positive solution $u_{\varepsilon}\in H^s_{V,\varepsilon}(\mathbb{R}^N)$ if $\varepsilon>0$ is small enough. Moreover, suppose $u_\varepsilon(x_\varepsilon) =\max\limits_{x\in\mathbb{R}^N} u_\varepsilon (x)$, then for each  $\rho > 0$,
\begin{align*}
  &{\lim\limits_{\varepsilon\to 0}V(x_{\varepsilon})} = \inf_{\Lambda}V,\\
  &{\liminf_{\varepsilon\to 0}dist(x_{\varepsilon},\mathbb{R}^N\backslash \Lambda)} > 0,\\
  &{\liminf_{\varepsilon\to 0}\|u_{\varepsilon}\|_{L^{\infty}(B_{\varepsilon \rho}(x_{\varepsilon}))}} > 0,\\
  &{\lim\limits_{{R\to \infty}\atop{\varepsilon \to 0}}\|u_{\varepsilon}\|_{L^{\infty}(\mathbb{R}^N\backslash B_{\varepsilon R}(x_{\varepsilon}))}} = 0.
\end{align*}
\end{theorem}

\begin{remark}\label{re1.2}
It seems that a natural  way to solve problem \eqref{eq1.1} is to use  the extension method developed by Caffarelli et al. in \cite{24}, by which, one can convert the nonlocal problem \eqref{eq1.1} into a local problem (see \cite{100} for example). But we do not take this method in our paper.  On one hand, if the problem \eqref{eq1.1} becomes a local problem, we believe that the difference between $0<s<1$ and $s = 1$ still makes it not easy to construct penalized function; On the other hand, we want to explore the essential properties behind the fractional Laplacian $(-\Delta)^s$. In fact, by our estimates on the lower bound of the penalized energy in Proposition \ref{pr3.3} and our construction of super-solutions and barrier functions in Section 4, we find that  all the difficulties are caused by the nonlocal operator $(-\Delta)^s$, which is a great contrast to the local operator $-\Delta$.
\end{remark}

\begin{remark}\label{re1.3}
To find solutions of \eqref{eq1.1}, it is natural to consider the following functional $I_{\varepsilon}:H^s_{V,\varepsilon}(\mathbb{R}^N)\to \mathbb{R}$:
$$
I_{\varepsilon}(u) = \frac{1}{2}\int_{\mathbb{R}^N}(\varepsilon^{2s}|(-\Delta)^{s/2}u|^2 + Vu^2)dx - \frac{1}{p}\int_{\mathbb{R}^N}|u|^pdx,
$$
whose critical points are solutions of equation \eqref{eq1.1}. However, the assumptions on $V$,  particularly the fact that $V$ can decay to zero at infinity, do not ensure that $I_{\varepsilon}$ is  well defined in $H^s_{V,\varepsilon}(\mathbb{R}^N)$. For example, let $V(x) = |x|^{-2ls/(l + 1)}$, $l > 0$, if $|x|\geq 1$. Take $u(x)\in C^{0,1}(\mathbb{R}^N)$ as follows
$$
u(x) = \left\{
  \begin{array}{ll}
    1, & x\in B_1(0),\vspace{2mm} \\
    \frac{1}{|x|^{\mu}} & x\in B^c_1(0),
  \end{array}
\right.
$$
with $\frac{N}{2}>\mu > \frac{N}{2} - \frac{ls}{l + 1}$, then $u\in H^s_{V,\varepsilon}(\mathbb{R}^N)$. But since $(N/2 - s)2^*_s = N$, we have
$$
\int_{\mathbb{R}^N}|u|^{p}dx = +\infty\ \text{for}\ l\ \text{large and}\ \mu \ \text{closed to}\ \frac{N}{2} - \frac{ls}{l + 1}.
$$
Moreover, even if we assume that $\inf_{\mathbb{R}^N}V > 0$, the functional would have a mountain-pass geometry in  $H^s_{V,\varepsilon}(\mathbb{R}^N)$, but the Palais-Smale condition could fail without further specific assumptions on $V$. To overcome this difficulty, we take the penalized ideas to cut off the nonlinearity, which was introduced  by Jean Van Schaftingen et al. in \cite{28}. We truncate the nonlinear term through a penalization outside the set $ \Lambda$ where the concentration is expected.

Generally speaking, the penalization argument is essentially a localization argument, while the operator $(-\Delta)^s$ is nonlocal. Hence we will face some new difficulties when we use the penalization argument to deal with nonlocal problems.  For the proof of concentration phenomena, it is necessary to separate the occurrences of two or more possible peaks of the penalized solutions $u_{\varepsilon}$, see conclusion of Proposition \ref{pr3.3}. But this step is difficult due to the interaction of those possible peaks since the nonlocal operator $(-\Delta)^s$ appears. To overcome  this difficulty, by employing  a skillful decomposition (see the simplification of equation \eqref{teq}), and then using the  assumption $(V_2)$ on $V$, we conclude that these interactions are small enough as $\varepsilon\to 0$, see \eqref{impotrant} in our proof of Proposition \ref{pr3.3}.

To prove that the solutions of penalized equation (Lemma \ref{le2.7}) are solutions of our origin problem \eqref{eq1.1}, a comparison principle for $(-\Delta)^s$ will be needed. Unlike local operator $-\Delta$, this kind of comparison principle requires the global information of the penalized solutions $u_{\varepsilon}$, see \eqref{eq4.12} in Section 4. Furthermore, as we can see in Proposition \ref{pr4.5}, due to the nonlocal effect of $(-\Delta)^s$, we have to construct supersolutions for the linearized penalized equation \eqref{eql4.2} globally. The construction is much more complicated than local case, since we can not see easily how the supersolutions would look like. For global information of $u_{\varepsilon}$, we refer to Appendix D in \cite{20}, where  $L^{\infty}$-norm and regularity assets are given. For construction of supersolutions, what puzzles us a lot is that we can not make the estimates  as precise  as those  of the local case $-\Delta$. We eventually solve this problem by the following essential observation: the solution for fractional equation  is always polynomial decay (see \cite{20} and so on). For more details we can see our proof of Propositions \ref{pr4.4} and \ref{pr4.5}.

\end{remark}

 We organize this paper as follows. In Section 2, we give a  variational framework and the penalized scheme and  show  the existence of solutions to the penalized problem via minimax-theorem. In Section 3, we discuss  the limiting equation corresponding to equation \eqref{eq1.1}, which plays a key role in excluding concentration occurring  at other points of $V$. In Section 4, we construct penalized function and the  corresponding barrier functions to show that the solution of penalized problem is  indeed  a solution of the original equation \eqref{eq1.1}. During our construction, the range $p\in (2 + \frac{2s}{N - 2s},2^*_s)$ and the decay assumption on $V$  are essential.

\vspace{1.2cm}

\section{The penalized problem}

\vspace{0.5cm}

\noindent The following inequality exposes the relationship between $H^s(\mathbb{R}^N)$ and the Banach space $L^q(\mathbb{R}^N)$.
\begin{proposition}\label{pr2.1} (Fractional version of the Gagliardo$-$Nirenberg inequality)
For every $u\in H^s(\mathbb{R}^N)$,
$$
\|u\|_{q} \leq C \|(-\Delta)^{s/2}u\|^{\beta}_{2}\|u\|^{1 - \beta}_{2},
$$
where $q\in [2,2^*_s]$ and $\beta$ satisfies $\frac{\beta}{2^*_s} + \frac{(1 - \beta)}{2} = \frac{1}{q}$.
\end{proposition}

\begin{remark}\label{re2.2}
The above inequality implies that $ H^s(\mathbb{R}^N)$ is continuously  embedded into $L^q(\mathbb{R}^N)$ for $\ q\in [2,2^*_s]$. Moreover, on bounded set, the embedding is compact ( see \cite{4}), i.e.,
$$
H^s(\mathbb{R}^N)\subset\subset L^q_{loc}(\mathbb{R}^N)\ \text{compactly, if}\ q\in [1,2^*_s).
$$
\end{remark}

\subsection{Definition of the penalized functional}

We choose a family of penalization potentials $P_{\varepsilon}\in L^{\infty}(\mathbb{R}^N,[0,\infty))$ for $\varepsilon > 0$ small in such a way that
\begin{equation}\label{2.1}
P_{\varepsilon}(x) = 0\ \text{for all}\ x\in\Lambda\ \text{and} \lim\limits_{\varepsilon \to 0}\sup_{\mathbb{R}^N\backslash \Lambda} P_{\varepsilon} = 0.
\end{equation}
We will give the explicit form of $P_{\varepsilon}$ in Sect. 4. Before that, we only rely on the following two assumptions on $P_{\varepsilon}$:

$(P_1)$ the space $H^s_{V,\varepsilon}(\mathbb{R}^N)\subset\subset L^2(\mathbb{R}^N,(P_{\varepsilon}(x) + \chi_{\Lambda}(x))dx)$ is compact,

$(P_2)$ there exists $0 < \tau < 1$ such that for  $\varphi\in H^s_{V,\varepsilon}(\mathbb{R}^N)$,
$$
\int_{\mathbb{R}^N}P_{\varepsilon}(x)|\varphi|^2 dx \leq \tau \int_{\mathbb{R}^N}\varepsilon^{2s}|(-\Delta)^{s/2} \varphi|^2 + V|\varphi|^2.
$$

    Given a penalization potential $P_{\varepsilon}$ which satisfies  $(P_1)$ and $(P_2)$,
we define the penalized nonlinearity $g_{\varepsilon}:\mathbb{R}^N\times \mathbb{R}\to \mathbb{R}$  by
$$
g_{\varepsilon}(x,s): = \chi_{\Lambda}(x)s^{p - 1}_{+} + \chi_{\mathbb{R}^N\backslash \Lambda}(x)\min(s^{p - 1}_{+}, P_{\varepsilon}(x)s_+).
$$
We also denote $G_{\varepsilon}(x,t) = \int_{0}^{t}g_{\varepsilon}(x,s)ds$.

Accordingly, we define the penalized superposition operators $\mathfrak{g}_{\varepsilon}$ and $\mathfrak{G}_{\varepsilon}$  by
$$
\mathfrak{g}_{\varepsilon}(u)(x) = g_{\varepsilon}(x,u(x))\ \text{and}\ \mathfrak{G}_{\varepsilon}(u)(x) = G_{\varepsilon}(x,u(x)),
$$
and the penalized functional $J_{\varepsilon}:H^s_{V,\varepsilon}(\mathbb{R}^N)\to \mathbb{R}$ by
$$
J_{\varepsilon}(u) = \frac{1}{2}\int_{\mathbb{R}^N}(\varepsilon^{2s}|(-\Delta)^{s/2} u|^2 + V(x)|u|^2) - \frac{1}{p}\int_{\mathbb{R}^N}\mathfrak{G}_{\varepsilon}(u)dx.
$$

\vspace{0.5cm}

\begin{remark}\label{re2.4}
Under assumption $(P_1)$, it is easy to check that the penalized functional $J_{\varepsilon}$ is well-defined. Moreover, $J_{\varepsilon}$ satisfies the Palais-Smale condition.
\end{remark}

%
%
%

\begin{lemma}\label{le2.5}

(1) If $2 < p< 2^*_s$ and the assumption $(P_1)$ holds, then $J_{\varepsilon}\in C^1(H^s_{V,\varepsilon}(\mathbb{R}^N), \mathbb{R})$ and for $u\in H^s_{V,\varepsilon}(\mathbb{R}^N)$, $\varphi\in H^s_{V,\varepsilon}(\mathbb{R}^N)$,
$$
\langle J_{\varepsilon}'(u),\varphi\rangle = \int_{\mathbb{R}^N}\varepsilon^{2s}(-\Delta)^{s/2}u(-\Delta)^{s/2}\varphi + Vu\varphi - \int_{\mathbb{R}^N}\mathfrak{g}_{\varepsilon}(u)\varphi dx.
$$

Here $\langle\cdot,\cdot\rangle$ denotes the duality product between the dual space $H^s_{V,\varepsilon}(\mathbb{R}^N)'$ and the space $H^s_{V,\varepsilon}(\mathbb{R}^N)$. In particular, $u\in H^s_{V,\varepsilon}(\mathbb{R}^N)$ is a critical point of $J_{\varepsilon}$ if and only if $u$ is a weak solution of the penalized equation
$$
\varepsilon^{2s}(-\Delta)^s u + Vu =\mathfrak{g}_{\varepsilon}(u).
\eqno(\mathcal{Q}_{\varepsilon})
$$

(2) (PS condition) If $2 < p <2^*_s$  and the assumptions $(P_1)$ and $(P_2)$ hold, then $J_{\varepsilon}$ owns the mountain pass geometry and satisfies the Palais-Smale condition.
\end{lemma}
\begin{proof}
To prove $(1)$, we only need to prove that the nonlinear term
$$
\mathcal{N}_{\varepsilon}(u) = \int_{\mathbb{R}^N}\mathfrak{G}_{\varepsilon}(u)
$$
is $C^1$. Indeed, assumption $(P_1)$ implies that for  $|t|\leq 1$,
$$
|\mathfrak{G}_{\varepsilon}(u + t\varphi) - \mathfrak{G}_{\varepsilon}(u)|\leq C\Big((|u|^p + |\varphi|^p)\chi_{\Lambda} + P_{\varepsilon}(|u|^2 + |\varphi|^2)\Big)\in L^1(\mathbb{R}^N),
$$
then the first order Gateaux derivative exists by Dominated Convergence Theorems. On the other hand, let $(u_n),\ u\in H^s_{V,\varepsilon}(\mathbb{R}^N)$ satisfy  $u_n\to u\in H^s_{V,\varepsilon}(\mathbb{R}^N)$. By Remark \ref{re2.2}, assumption $(P_1)$, and Dominated Convergence Theorem, we deduce that  for $\varphi\in H^{s}_{V,\varepsilon}(\mathbb{\mathbb{R}^N})$ with $\|\varphi\|_{H^{s}_{V,\varepsilon}(\mathbb{R}^N} \leq 1$,
\begin{align*}
 &\quad\quad|\langle\mathcal{N}'_{\varepsilon}(u_n) - \mathcal{N}'_{\varepsilon}(u),\varphi\rangle|&\\
&\leq C\|u_n - u\|^p_{H^{s}_{V,\varepsilon}(\mathbb{R}^N)} + \int_{\mathbb{R}^N\backslash{\Lambda}}\big|\min\{P_{\varepsilon},(u_n)^{p - 2}_+\}u_n\varphi - \min\{P_{\varepsilon},u^{p - 2}_+\}u\varphi \big| dx\\
 & \leq  o_n(1) + \int_{\mathbb{R}^N\backslash{\Lambda}}\big|\min\{P_{\varepsilon},(u_n)^{p - 2}_+\}(u_n - u)\varphi\big|\\
 &\ \ \ \ \ \ \ \ \ \ \ \ \ \ \qquad + \big|\min\{P_{\varepsilon},(u_n)^{p - 2}_+\} - \min\{P_{\varepsilon},u^{p - 2}_+\}\big|u\varphi dx\\
 & \leq o_n(1) + \int_{\mathbb{R}^N\backslash{\Lambda}}\big|\min\{P_{\varepsilon},(u_n)^{p - 2}_+\} - \min\{P_{\varepsilon},u^{p - 2}_+\}\big||u|^2dx\\
 & = o_n(1).
\end{align*}

To prove $(2)$, it is easy to check the mountain pass geometry. Now we show that $J_{\varepsilon}$ satisfies the Palais-Smale condition, i.e., any sequence $(u_n)\in H^s_{V,\varepsilon}(\mathbb{R}^N)$ satisfying  $\sup_n J_{\varepsilon}(u_n) \leq C < \infty$ and  $\lim\limits_{n\to\infty} J'_{\varepsilon}(u_n)\to 0\ \text{in}\ (H^s_{V,\varepsilon}(\mathbb{R}^N))'$ must be relatively compact.

     Indeed, $(P_2)$ implies that
\begin{align*}
 &C + o_n(1)\|u_n\|_{H^s_{V,\varepsilon}(\mathbb{R}^N)}\\
 \geq & \Big(\frac{1}{2} - \frac{1}{p}\Big)\|u_n\|_{H^s_{V,\varepsilon}(\mathbb{R}^N)} + \Big(\frac{1}{p} - \frac{1}{2}\Big)\int_{\mathbb{R}^N}P_{\varepsilon}(x)(u_n)^2 dx\\
 \geq & \Big(\frac{1}{2} - \frac{1}{p}\Big)\|u_n\|_{H^s_{V,\varepsilon}(\mathbb{R}^N)}  - \tau\Big(\frac{1}{2} - \frac{1}{p}\Big)\|u_n\|_{H^s_{V,\varepsilon}(\mathbb{R}^N)}\geq C\|u_n\|^2_{H^s_{V,\varepsilon}(\mathbb{R}^N)}.
\end{align*}
Hence $(u_n)$ is bounded in $H^{s}_{V,\varepsilon}(\mathbb{R}^N)$. Then, up to a subsequence, we assume that $u_n\rightharpoonup u\in H^s_{V,\varepsilon}(\mathbb{R}^N)$ weakly. By assumptions $(P_1)$ and Remark \ref{re2.2}
\begin{align*}
  \|u_n - u\|^2_{H^s_{V,\varepsilon}(\mathbb{R}^N)}
 = &\langle J'_{\varepsilon}(u_n) - J'_{\varepsilon}(u),u_n - u \rangle+ \int_{\mathbb{R}^N}(\mathfrak{g}_{\varepsilon}(u_n) - \mathfrak{g}_{\varepsilon}(u))(u
_n(x) - u(x))dx\\
 \leq & o_n(1) + C\Big(\int_{\mathbb{R}^N\backslash \Lambda}P_{\varepsilon}(|u_n| + |u|)^2dx\Big)^\frac{1}{2}\Big(\int_{\mathbb{R}^N\backslash \Lambda}P_{\varepsilon}|u_n-u|^2dx\Big)^\frac{1}{2}\\
 =& o_n(1).
\end{align*}
As a result,  we complete  the proof.
\end{proof}

Define
$$
c_{\varepsilon}: = \inf_{\gamma\in \Gamma_{\varepsilon}}\max_{t\in [0,1]}J_{\varepsilon}(\gamma(t)) > 0,
$$
where
$$
\Gamma_{\varepsilon}:= \{\gamma\in C([0,1],H^s_{V,\varepsilon}(\mathbb{R}^N))|\gamma(0) = 0,\ J_{\varepsilon}(\gamma(1)) < 0\}.
$$

Now we have
\begin{lemma}\label{le2.7}
There exists $ u_{\varepsilon} \in H^s_{V,\varepsilon}(\mathbb{R}^N)\cap C^{1,\beta}(\mathbb{R}^N)$ for some $\beta\in (0,1)$ such that $J_{\varepsilon}(u_{\varepsilon}) = c_{\varepsilon}$,
 $J'_{\varepsilon}(u_{\varepsilon}) = 0$, and $u_{\varepsilon} > 0$.
\end{lemma}

\begin{proof}
The existence of $u_{\varepsilon}$ follows from Lemma~\ref{le2.5}. The regularity  result follows from Appendix D in \cite{20}. Testing the penalized equation $(\mathcal{Q}_{\varepsilon})$ with $(u_{\varepsilon})_{-}$ and integrating, we can see that $u_{\varepsilon}\geq 0$.

Suppose to the contrary that there exists  $x_0\in \mathbb{R}^N$ such that $u_{\varepsilon}(x_0) = 0$, then we have
$$
0 = (-\Delta)^s{u_{\varepsilon}}(x_0) + V(x_0)u_{\varepsilon}(x_0) < 0,
$$
which is a contradiction. Therefore, $u_{\varepsilon}>0$.

\end{proof}

\section{Concentration of the penalized solutions}

\noindent In this section, we will compare the energy $c_\varepsilon$ obtained by Lemma \ref{le2.7} with that of the following limiting problem, which is useful to  get the concentrating phenomena of the  solutions obtained by Lemma \ref{le2.7}.

\subsection{The limiting problem}

\noindent For $\lambda > 0$, the limiting problem associated to  \eqref{eq1.1} is
$$
(-\Delta)^s v + \lambda v = |v|^{p - 2}v,
\eqno(\mathcal{P}_{\lambda})
$$
and the corresponding functional $J_{\lambda}: H^s(\mathbb{R}^N)\to \mathbb{R}$ defined  by
$$
J_{\lambda}(v) = \frac{1}{2}\int_{\mathbb{R}^N}(|(-\Delta)^{s/2} v|^2 + \lambda |v|^2)dx - \frac{1}{p}\int_{\mathbb{R}^N}|v|^p dx.
$$

As we introduced in the introduction, the limiting problem has a positive ground state. We define the limiting energy by
\begin{equation}\label{eq3.1}
\mathcal{E}(\lambda) = \inf_{v\in H^s(\mathbb{R}^N)\backslash\{0\}}\max_{t > 0}I_{\lambda}(tv).
\end{equation}
Since for every $v\in H^s(\mathbb{R}^N),\ J_{\lambda}(|v|) \leq J_{\lambda}(v)$, by the  density of $C_0^\infty(\mathbb{R}^N)$ in   $H^s(\mathbb{R}^N)$, we have
\begin{equation}\label{eq3.2}
  \mathcal{E}(\lambda) = \inf_{{C^{\infty}_c(\mathbb{R}^N)\backslash\{0\}}\atop{v\geq 0}}\max_{t > 0}J_{\lambda}(tv).
\end{equation}
The following proposition is obvious.
\begin{proposition}\label{pr3.1}
Let $\lambda > 0$ and $v\in H^s(\mathbb{R}^N)$. Define
$$
v_{\lambda}(y) = \lambda^{\frac{1}{p - 2}}v(\lambda^{1/2s}y).
$$
Then
$$
J_{\lambda}(v_{\lambda}) = \lambda^{\frac{p}{p - 2} - \frac{N}{2s}}J_1(v_1).
$$
In particular, $v$ is a solution of $(\mathcal{P}_1)$  if and only if $v_\lambda$ is a solution of $(\mathcal{P}_{\lambda})$ and
$$
\mathcal{E}({\lambda}) = \mathcal{E}({1})\lambda^{\frac{p}{p - 2} - \frac{N}{2s}}.
$$
Moreover, $\mathcal{E}({\lambda})$ is continuous and increasing in $(0,\infty)$.
\end{proposition}

Define the concentration function $\mathcal{C}:\ \mathbb{R}^N\to \mathbb {R}$ by
\begin{equation}\label{eq3.3}
 \mathcal{C}(x) = \mathcal{E}(V(x)).
\end{equation}

\vspace{1cm}

\noindent We begin our study of the asymptotic behavior of solutions by an upper bound on the mountain-pass level.

\begin{lemma}\label{le3.2}(Upper bound of the energy) we have
\begin{equation}\label{eq3.4}
 \limsup_{\varepsilon\to 0}\frac{c_{\varepsilon}}{\varepsilon^N}\leq \inf_{x\in\Lambda}\mathcal{C}(x).
\end{equation}
\end{lemma}

\begin{proof}
For any given $a \in \Lambda$ and nonnegative function $v \in C^{\infty}_c (\mathbb{R}^N)\backslash\{0\}$, we define
$$
v_{\varepsilon}(x) = v\Big(\frac{x - a}{\varepsilon}\Big).
$$
Obviously, $supp \,\,v \subset \Lambda$ when $\varepsilon$ is small enough. It is easy  to check that $\gamma_{\varepsilon}(t) = tT_0v_{\varepsilon}\in \Gamma_{\varepsilon}$ for some $T_0$ large enough. So
\begin{align*}
  \frac{c_{\varepsilon}}{\varepsilon^N}&\leq \max_{t\in [0,1]}J_{\varepsilon}(\gamma_{\varepsilon}(t))\leq \max_{t>0}I_{V(a)}(tv) + o_\varepsilon(1),
\end{align*}
and
$$
\limsup_{\varepsilon\to 0}\frac{c_{\varepsilon}}{\varepsilon^N}\leq \inf_{{C^{\infty}_c(\mathbb{R}^N)\backslash\{0\}}\atop{v\geq 0}}\max_{t > 0}I_{\lambda}(tv) \leq \mathcal{E}(V(a)) = \mathcal{C}(a),
$$
which implies our conclusion.
\end{proof}

\vspace{0.5cm}

In order to understand the behavior  of the solutions we will rely on a  lower bound on the energy.

We first establish  a Liouville type  theorem for fractional problems on a half-space, from which we can expand the penalized problem to a limiting equation defined on the hole space $\mathbb{R}^N$.

\begin{lemma}\label{le3.4}
Let $H \subset \mathbb{R}^N$ be a half-space. If $v \in H^s_{V,\varepsilon}(\mathbb{R}^N), \ v\geq 0$ satisfies the equation
$$
(-\Delta)^s v + \lambda v = \chi_{H}v^{p - 1}\ \text{in}\ \mathbb{R}^N,
$$
then $v = 0$.
\end{lemma}

\begin{proof}
Without loss of generality, we assume $H = \{(x',x_N):x'\in \mathbb{R}^{N - 1},\ x_N\geq 0\}$. Let $x_0 = (0,\cdots,0,-2)$ and define
$$
\eta_{\sigma}(x): = \left\{
                      \begin{array}{ll}
                        C\sigma^{-N}exp\Big(\frac{1}{|x/\sigma|^2 - 1}\Big), & \text{if}\ |x| < \sigma, \\
                        0, & \text{if}\ |x|\geq \sigma,\
                      \end{array}
                    \right.
$$
where $0<\sigma<\frac{1}{3}$ and the constant $C$ is selected so that $\int_{\mathbb{R}^N}\eta_{\sigma} = 1$ . Let $v_{\sigma} = (\eta_{\sigma}\ast\chi_{B_1(x_0)})v$, then $v\in H^s(\mathbb{R}^N)$. Testing the equation
 with $v_{\sigma}$, to find
$$
\int_{\mathbb{R}^N}(-\Delta)^sv(x)v_{\sigma}(x)dx + \int_{\mathbb{R}^N}vv_{\sigma} = 0.
$$
Letting $\sigma\to 0$, and by Dominated Convergence Theorem, we have
$$
\int_{B_1(x_0)}|(-\Delta)^{s/2}v|^2dx + \int_{B_1(x_0)}v^2 = 0.
$$
Then $v = 0$ a.e. on $B_1(x_0)$. By the regularity result (see Appendix D in \cite{20}), we can write
$$
(-\Delta)^sv(x_0) + \lambda v(x_0)= \int_{\mathbb{R}^N}\frac{-v(y)}{|x - y|^{N + 2s}},
$$
which is negative if $v > 0$ at somewhere. Thus we conclude that $v(x)\equiv 0$ in $\mathbb{R}^N$.
\end{proof}

    The following proposition compares the penalized energy $c_{\varepsilon}$ with the limiting energy $\mathcal{C}(\cdot)$, which is important  to analyze the concentration behavior of  $u_{\varepsilon}$. A key method of the proof is truncating skill, which is necessary since $V$ may vanish at infinity  (see \eqref{teq0} and \eqref{teq}). However, as we can see in our estimates \eqref{eq3.7} and \eqref{eq3.8}, the truncating procedure makes us have to estimate the penalized energy outside some balls, which  is more complicated  in the nonlocal case.

\begin{proposition}\label{pr3.3}
 Let $(\varepsilon_n)_{n\in \mathbb{N}}$ be a sequence of positive numbers converging to $0$,  $(u_n)$ be a sequence of critical points given by Lemma \ref{le2.7},
  and for $j\in {1,2,\cdots,k}$,  $(x^j_n)$ be a sequence in $\mathbb{R}^N$  converging  to $x^j_{*}\in \mathbb{R}^N$. If
$$
\limsup_{n\to\infty}\frac{1}{\varepsilon^N_n}\int_{\mathbb{R}^N}\varepsilon^{2s}_n|(-\Delta)^{s/2}u_n|^2 + V|u_n|^2 < \infty,
$$

$$
V(x^j_{*}) > 0\ \text{and}\ \lim\limits_{n\to \infty}\frac{|x^i_n - x^j_n|}{\varepsilon_n} = \infty\ i\neq j\ \text{for all}\ i,j = 1,2,\cdots,k,
$$
and for some $\rho > 0$,
$$
\liminf_{n\to \infty}\|u_n\|_{L^{\infty}(B_{\varepsilon_n\rho}(x^j_n))} > 0\ j = 1,2,\cdots k,
$$
then $x^j_{*}\in \bar{\Lambda}$ and
$$
\liminf_{n\to \infty}\frac{J_{\varepsilon_n}(u_n)}{{\varepsilon^N_n}}\geq\sum_{j = 1}^{k}\mathcal{C}(x^j_*).
$$
\end{proposition}

\begin{proof}
 For $j\in \{1,\ldots,k\}$, we define  the rescaling  function $v^j_n\in H^s_{loc}(\mathbb{R}^N)$ by
$$
v^j_n(y) = u_n(x^j_n + \varepsilon_ny).
$$
Since $u_n$ solves the penalized problem $(\mathcal{Q}_{\varepsilon})$, the function $v^j_n$ satisfies weakly the rescaled equation
$$
(-\Delta)^{s} v^j_n + V^j_n v^j_n = \textsf{g}^j_n(v^j_n)\ \text{in}\ \mathbb{R}^N,
$$
where  $V^j_n(y) = V(x^j_n + \varepsilon_n y),\ \textsf{g}^j_n(v^j_n) = g_{\varepsilon_n}(x^j_n + \varepsilon_ny, v^j_n).
$

For every $R > 0$, we see
$$
\int_{B_R}|(-\Delta)^{s/2}v^j_n|^2 + V^j_n|v^j_n|^2 = \frac{1}{\varepsilon^N_n}\int_{B_{\varepsilon_nR}(x^j_n)}\varepsilon^{2s}_n|(-\Delta)^{s/2}u_n|^2 + V|u_n|^2.
$$
Since $V$ is positive and continuous at $x^j_* = \lim\limits_{n\to\infty}x^i_n$, we have
\begin{equation}\label{teq0}
\liminf_{n\to\infty}\int_{B_R}|(-\Delta)^{s/2}v^j_n|^2 + V(x^j_*)|v^j_n|^2\leq\limsup_{n\to\infty}\frac{1}{\varepsilon^N_n}\int_{\mathbb{R}^N}\varepsilon^{2s}_n||(-\Delta)^{s/2}u_n|^2 + V|u_n|^2.
\end{equation}
Hence, via a diagonal argument, there exists $v^j_*\in H^s_{loc}(\mathbb{R}^N)$ such that for every $R > 0$, $v^j_n\rightharpoonup v^j_*$ weakly in $H^s(B_R)$  and
\begin{align*}
&\int_{B_R}dx\Big(\int_{B_R}\frac{|v^j_{*}(x) - v^j_{*}(y)|^2}{|x - y|^{N + 2s}}dy + V(x^j_*)|v^j_*|^2\Big)\\
              \leq& \liminf_{n\to\infty}\frac{1}{\varepsilon^N_n}\int_{\mathbb{R}^N}\varepsilon^{2s}_n|(-\Delta)^{s/2}u_n|^2 + V|u_n|^2 < \infty.
\end{align*}
This implies $v^j_*\in H^s(\mathbb{R}^N)$. Moreover, for  $1\leq q < 2^*_s$, we have $v^j_n \to v^j_*$ in $L^q(B_R)$ strongly as $n\to\infty$.

Define
$$
\Lambda^j_n = \{y\in \mathbb{R}^N: x^j_n + \varepsilon_n y \in \Lambda\}.
$$
Since $\partial\Lambda$ is smooth, we can assume that $\chi_{\Lambda^j_n}\to \chi_{\Lambda^j_*}$ almost everywhere as $n\to\infty$, where $\Lambda^j_*$ is either $\mathbb{R}^N$, a half-space or $\emptyset$.

Observe that
\begin{align*}
  g^j_n(x,t)  & = \chi_{\Lambda^j_n}(x)t^{p - 1}_+ + \chi_{\mathbb{R}^N\backslash \Lambda^j_n}(x)\min\big\{P_{\varepsilon_n}(x)t_+, t^{p - 1}_+\big\}\vspace{2mm}\\
              &\rightarrow \chi_{\Lambda^j_*}(x)t^{p - 1}_+\ \forall\,\,(x,t)\in \mathbb{R}^N\times \mathbb{R},
\end{align*}
and
$$
|\textsf{g}^j_n(v^j_n)|\leq (v^j_n)^{p - 1}_+.
$$
Then by Remark \ref{re2.2},
\begin{equation}\label{eq3.5}
\textsf{g}^j_n(v^j_n) \to \chi_{\Lambda^j_*}(x)(v^j_*)^{p - 1}_+\ \text{in}\ L^q_{loc}(\mathbb{R}^N),
\end{equation}
where $ q(p - 1) < 2^*_s$. Hence  $v^j_*$ satisfies
$$
(-\Delta)^sv^j_* + V(x^j_*)v^j_* = \chi_{\Lambda^j_*}(v^j_*)^{p - 1}_+,\,\,x\in \Lambda^j_*.
$$

By the  bootstrap argument adapted to the fractional Schr\"{o}dinger equation (see  Appendix D in \cite{20}, for example), $v^j_n\to v^j_*$ uniformly on compact subsets of $\mathbb{R}^N$. Hence
\begin{align*}
 &\|v^j_*\|_{L^{\infty}(B_\rho)}
  = \lim\limits_{n\to\infty}\|v^j_n\|_{L^{\infty}(B_\rho)}
  = \lim\limits_{n\to\infty}\|u_n\|_{L^{\infty}(B_{\varepsilon_n\rho}(x^j_n))} > 0. \\
\end{align*}

Therefore, $v^j_*\not\equiv 0$ and the set $\Lambda^j_*$ cannot be empty. By Lemma \ref{le3.4}, it can neither be half-space. Hence we conclude that $\Lambda^j_* = \mathbb{R}^N$, $x^j_*\in \bar{\Lambda}$, and the function $v^j_*$ satisfies
$$
(-\Delta)^{s}v^j_* + V(x^j_*)v^j_* = (v^j_*)^{p - 1}_+\ \  \text{in}\ \mathbb{R}^N.
$$
Since $v^j_* \geq 0$, we see
\begin{align*}
  &\liminf_{n\to\infty}\frac{1}{{\varepsilon^N_n}}\int_{B_{\varepsilon_nR}(x^j_n)}\frac{1}{2}(\varepsilon^{2s}_n|(-\Delta)^{s/2}u_n|^2 + V|u_n|^2) - \mathfrak{G}_{\varepsilon_n}(u_n)\\
  \geq & \frac{1}{2}\int_{B_R}|(-\Delta)^{s/2}v^j_*|^2 + V(x^j_*)|v^j_*|^2 - \int_{\mathbb{R}^N}\frac{1}{p}(v^j_*)^p_+\\
  &- \frac{1}{2}\int_{B_R}dx \int_{B^c_R}\frac{|v^j_*(x) - v^j_*(y)|^2}{|x - y|^{N + 2s}}dy\\
 \geq &\mathcal{C}(x^j_*) - C\int_{\mathbb{R}^N\backslash B_R}|(-\Delta)^{s/2}v^j_*|^2 + V(x^j_*)|v^j_*|^2,
\end{align*}
where $C > 0$ is a constant.

In order to study the integral outside $B_{\varepsilon_nR}(x^j_n)$, we choose $\eta\in C^{\infty}(\mathbb{R}^N)$ such that $0\leq \eta\leq 1,\ \eta = 0$ on $B_1$ and $\eta = 1$ on $\mathbb{R}^N\backslash B_2$. Define
$$
\psi_{n,R}(x) = \prod^{k}_{j = 1}\eta\Big(\frac{x - x^j_n}{\varepsilon_n R}\Big).
$$
Since $u_n$ is a solution to the penalized problem $(\mathcal{P}_{\varepsilon_n})$, we have by taking $\psi_{n,R}u_n$ as a test function in the penalized problem $(\mathcal{P}_{\varepsilon_n})$ that
\begin{align}\label{teq}
 \nonumber &\int_{\mathbb{R}^N\backslash\cup^{k}_{j = 1}B_{\varepsilon_nR}(x^j_n)}\varepsilon^{2s}_n\psi_{n,R}|(-\Delta)^{s/2}u_n|^2 + V\psi_{n,R}|u_n|^2\\
   = &\int_{\mathbb{R}^N\backslash\cup^{k}_{j = 1}B_{\varepsilon_nR}(x^j_n)}\mathfrak{g}_{\varepsilon_n}(u_n)u_n\psi_{n,R}\\
\nonumber  &-\int_{\mathbb{R}^N\backslash\cup^{k}_{j = 1}B_{\varepsilon_nR}(x^j_n)}\int_{\mathbb{R}^N}\frac{u_n(y)(\psi_{n,R}(x) - \psi_{n,R}(y))(u_n(x) - u_n(y))}{|x - y|^{N + s}}dydx\\
\nonumber  :=& \int_{\mathbb{R}^N\backslash\cup^{k}_{j = 1}B_{\varepsilon_nR}(x^j_n)}\mathfrak{g}_{\varepsilon_n}(u_n)u_n\psi_{n,R} + I_n.
\end{align}
Hence
\begin{align*}
  &\ds\int_{\mathbb{R}^N\backslash\cup^{k}_{j = 1}B_{\varepsilon_nR}(x^j_n)}\frac{1}{2}(\varepsilon^{2s}_n|(-\Delta)^{s/2}u_n|^2 + V|u_n|^2) - \mathfrak{G}_{\varepsilon_n}\\
  \geq & \frac{1}{2}\int_{\mathbb{R}^N\backslash\cup^{k}_{j = 1}B_{\varepsilon_nR}(x^j_n)}\varepsilon^{2s}_n\psi_{n,R}|(-\Delta)^{s/2}u_n|^2 + V\psi_{n,R}|u_n|^2 - \mathfrak{g}_{\varepsilon_n}(u_n)u_n\\
   =& -\frac{\varepsilon^{2s}_n}{2}I_n + \int_{\mathbb{R}^N\backslash\cup^{k}_{j = 1}B_{\varepsilon_nR}(x^j_n)}\mathfrak{g}_{\varepsilon_n}(u_n)(\psi_{n,R} - 1)u_n.
\end{align*}
Let $\eta_R(x) = \eta(\frac{x}{R})$, by scaling, if $n$ is large enough so that $B_{\varepsilon_nR}(x^j_n)\cap B_{\varepsilon_nR}(x^l_n) = \emptyset$,
\begin{align*}
&\frac{1}{\varepsilon^N_n}\int_{\mathbb{R}^N\backslash\cup^{k}_{j = 1}B_{\varepsilon_nR}(x^j_n)}\frac{1}{2}(\varepsilon^{2s}_n|(-\Delta)^{s/2}u_n|^2 + V|u_n|^2) - \mathfrak{G}_{\varepsilon_n}\\
\geq & - \frac{1}{2}\varepsilon^{2s - N}I_n -\frac{1}{2}\sum_{j = 1}^{k}\int_{B_{2R}\backslash B_{R}}\textsf{g}^j_n(v^j_n)v^j_n.
\end{align*}
By the fact that $v^j_n\to v^j_*$ in $L^{p}_{loc}(\mathbb{R}^N)$, we have
$$
\lim\sup_{n\to\infty}\Big|\frac{1}{2}\sum_{j = 1}^{k}\int_{B_{2R}\backslash B_{R}}\textsf{g}^j_n(v^j_n)v^j_n\Big| = \Big|\frac{1}{2}\sum_{j = 1}^{k}\int_{B_{2R}\backslash B_{R}}(v^j_*)^p\Big| = o_R(1).
$$
Now we estimate $I_n$. A change of variable tells us
\begin{align}\label{eqi1}
I_n & = \int_{\mathbb{R}^N}u_n(y)dy\int_{\mathbb{R}^N}\frac{(u_n(x) - u_n(y))(\psi_{n,R}(x) - \psi_{n,R}(y))}{|x - y|^{N + 2s}}dx\\
\nonumber    & = \varepsilon^{N - 2s}_n\sum_{l = 1}^k\int_{\mathbb{R}^N}v^l_n(y)\beta^l_n(y)dy\int_{\mathbb{R}^N}\frac{\alpha^l_n(x)(v^l_n(x) - v^l_n(y))(\eta_R(x) - \eta_R(y))}{|x - y|^{N + 2s}}dx,
\end{align}
where
$$
\beta^l_n(y) = \prod^{l - 1}_{s = 0}\eta\Big(\frac{y}{R} + \frac{x^l_n - x^s_n}{\varepsilon_n R}\Big),\,\,\,\beta^0_n(y)\equiv 1
$$
and
$$
\alpha^l_n(x) = \prod^{k}_{s = l + 1}\eta\Big(\frac{x}{R} + \frac{x^l_n - x^s_n}{\varepsilon_n R}\Big),\,\,\, \alpha^k_n(x) \equiv 1.
$$
Then
\begin{align*}
\varepsilon^{2s - N}_nI_n & = \sum_{l = 1}^k\int_{B_{2R}} v^l_n(y) \beta^l_n(y) dy \int_{B^c_{2R}}\frac{\alpha^l_n(x)(v^l_n(x) - v^l_n(y))(\eta_R(x) - \eta_R(y))}{|x - y|^{N + 2s}}dx\\
   &\quad + \sum_{l = 1}^k \int_{B^c_{2R}}v^l_n(y)\beta^l_n(y)dy  \int_{B_{2R}}\frac{\alpha^l_n(x)(v^l_n(x) - v^l_n(y))(\eta_R(x) - \eta_R(y))}{|x - y|^{N + 2s}}dx\\
   &\quad + \sum_{l = 1}^k\int_{B_{2R}}v^l_n(y)\beta^l_n(y)dy   \int_{B_{2R}}\frac{\alpha^l_n(x)(v^l_n(x) - v^l_n(y))(\eta_R(x) - \eta_R(y))}{|x - y|^{N + 2s}}dx\\
   &:=I^{(1)}_n + I^{(2)}_n + I^{(3)}_n.
\end{align*}
By the choice of $\eta$ and  $\lim\limits_{n\to\infty}\frac{|x^l_n - x^s_n|}{\varepsilon_n} = \infty$ if $l\neq s$, for $n$ large,
\begin{align*}
I^{(1)}_n&= \sum_{l = 1}^k\int_{B_{2R}\backslash B_R}v^l_n(y)\beta^l_n(y)dy\int_{B^c_{2R}}\frac{\alpha^l_n(x)(v^l_n(x) - v^l_n(y))(1 - \eta_R(y))}{|x - y|^{N + 2s}}dx\\
       &\quad + \sum_{l = 1}^k\int_{{B_R}}v^l_n(y)\beta^l_n(y)dy\int_{B^c_{2R}}\frac{\alpha^l_n(x)(v^l_n(x) - v^l_n(y))}{|x - y|^{N + 2s}}dx\\
       &= \sum_{l = 1}^k\int_{B_{2R}\backslash B_R}v^l_n(y)dy\int_{B^c_{2R}}\frac{\alpha^l_n(x)(v^l_n(x) - v^l_n(y))(1 - \eta_R(y))}{|x - y|^{N + 2s}}dx\\
       &\quad + \sum_{l = 1}^k\int_{{B_R}}v^l_n(y)dy\int_{B^c_{2R}}\frac{\alpha^l_n(x)(v^l_n(x) - v^l_n(y))}{|x - y|^{N + 2s}}dx\\
       :&=I^{(11)}_n + I^{(12)}_n
\end{align*}
and
\begin{align*}
        I^{(2)}_n
          =& \sum_{l = 1}^k\int_{B^c_{2R}}v^l_n(y)\beta^l_n(y)dy\int_{B_{2R}\backslash B_R}\frac{\alpha^l_n(x)(v^l_n(x) - v^l_n(y))(\eta_R(x) - 1)}{|x - y|^{N + 2s}}dx\\
         & - \sum_{l = 1}^k\int_{B^c_{2R}}v^l_n(y)\beta^l_n(y)dy\int_{B_R}\frac{\alpha^l_n(x)(v^l_n(x) - v^l_n(y))}{|x - y|^{N + 2s}}dx.\\
         =& \sum_{l = 1}^k\int_{B^c_{2R}}v^l_n(y)\beta^l_n(y)dy\int_{B_{2R}\backslash B_R}\frac{(v^l_n(x) - v^l_n(y))(\eta_R(x) - 1)}{|x - y|^{N + 2s}}dx\\
         & - \sum_{l = 1}^k\int_{B^c_{2R}}v^l_n(y)\beta^l_n(y)dy\int_{B_R}\frac{v^l_n(x) - v^l_n(y)}{|x - y|^{N + 2s}}dx.\\
         :=&I^{(21)}_n + I^{(22)}_n.
\end{align*}
Also, for large $n$,
\begin{align*}
I^{(3)}_n =&\sum_{l = 1}^k\int_{B_{2R}\backslash{B_R}}v^l_n(y)dy\int_{B_{2R}}\frac{(v^l_n(x) - v^l_n(y))(\eta_R(x) - \eta_R(y))}{|x - y|^{N + 2s}}dx\\
& + \sum_{l = 1}^k\int_{B_{2R}}v^l_n(y)dy\int_{B_{2R}\backslash{B_R}}\frac{(v^l_n(x) - v^l_n(y))(\eta_R(x) - \eta_R(y))}{|x - y|^{N + 2s}}dx\\
&+ \sum_{l = 1}^k\int_{B_{2R}\backslash{B_R}}v^l_n(y)dy\int_{B_{2R}\backslash B_{R}}\frac{(v^l_n(x) - v^l_n(y))(\eta_R(x) - \eta_R(y))}{|x - y|^{N + 2s}}dx\\
:=& I^{(31)}_n + I^{(32)}_n + I^{(33)}_n.
\end{align*}

By the assumption on $V$, we have
\begin{align}\label{eq3.7}
\nonumber    &\limsup_{n\to\infty}  |I^{(i2)}_n|   \\
\nonumber    \leq& 2\sum_{l = 1}^k\int_{B^c_{2R}}dy\int_{B_{R}}\frac{(v^l_n(x))^2 + (v^l_n(y))^2}{(|y| - R)^{N + 2s}}dx\\
                     \leq & 2\sum_{l = 1}^k\int_{B_{R}}(v^l_n(x))^2dx \int_{B^c_{2R}}\frac{1}{(|y| - R)^{N + 2s}}dy\\
\nonumber& + 2C_NR^{N}\sum_{l = 1}^k\int_{B^c_{2R}}\frac{V^l_n(y)(v^l(y))^2}{V^l_n(y)|{\varepsilon_n} y + x^j_n|^{2s}|{\varepsilon_n} y + x^j_n|^{-2s}|y|^{2s}|y|^{-2s}
(|y| - R)^{N + 2s}}dy\\
\nonumber\leq & CR^{-2s}\sum_{l = 1}^k\int_{B_{R}}(v^l_*(x))^2dx + C\limsup_{n\to\infty}\sum_{l = 1}^k\Big(\varepsilon_n + \frac{1}{R}\Big)^{2s}\int_{B^c_{2R}}V^l_n(y)(v^l(y))^2\\
\nonumber =& o_R(1)
\end{align}
for $i = 1,2$.

Now for $I^{(11)}_n$, since $v^l_n\rightharpoonup v^l_*$ in $H^s_{loc}(\mathbb{R}^N)$, $v^l_n\to v^l_*$ in $L^q_{loc}(\mathbb{R}^N)$ if $1\leq q < 2^*_s$, for $1\leq l\leq k$. By Cauchy inequality and the estimates of $I^{(i2)}_n$,
\begin{eqnarray}\label{eq3.8}
&&\limsup_{n\to\infty}  |I^{(11)}_n|\\
\nonumber&\leq& \limsup_{n\to\infty}  \sum_{l = 1}^k\int_{B_{2R}\backslash B_R}dy\int_{B^c_{2R}}\frac{|v^l_n(x) - v^l_n(y)|^2}{|x - y|^{N + 2s}}dx\\
\nonumber && + \limsup_{n\to\infty}\sum_{l = 1}^k \int_{B_{2R}\backslash B_R}(v^l_n(y))^2dy \int_{B^c_{2R}}\frac{(1 - \eta_{R}(y))^2}{|x - y|^{N + 2s}}dx\\
\nonumber & \leq &\limsup_{n\to\infty}  \sum_{l = 1}^k  \int_{B_{2R}\backslash B_R}dy\int_{B^c_{2R}\cap B_{4R}}\frac{|v^l_n(x) - v^l_n(y)|^2}{|x - y|^{N + 2s}}dx\\
\nonumber &&+ \limsup_{n\to\infty}\sum_{l = 1}^k \int_{B_{2R}\backslash B_R}dy\int_{B^c_{4R}}\frac{|v^l_n(x) - v^l_n(y)|^2}{|x - y|^{N + 2s}}dx + C\limsup_{n\to\infty}\sum_{l = 1}^k\int_{B_{2R}\backslash B_R}(v^l_n(y))^2dy\\
\nonumber & \leq& \limsup_{n\to\infty}\sum_{l = 1}^k\int_{B_{2R}\backslash B_R}dy\int_{B^c_{2R}\cap B_{4R}}\frac{|v^l_n(x) - v^l_n(y)|^2}{|x - y|^{N + 2s}}dx + C \limsup_{n\to\infty}(\varepsilon_n + 1/R)^{2s}\\
\nonumber &&+ C\limsup_{n\to\infty}\sum_{l = 1}^k\int_{B_{2R}\backslash B_R}(v^l_n(y))^2dy\\
\nonumber & \leq& C\int_{B_{2R}\backslash B_R}(v^l_*(y))^2dy + o_R(1) \\
\nonumber &&+ \limsup_{n\to\infty}\sum_{l = 1}^k\int_{B_{4R}\cap B^c_{3R/2}}dy\int_{B_{4R}\cap B^c_{3R/2}}\frac{|v^l_n(x) - v^l_n(y)|^2}{|x - y|^{N + 2s}}dx\\
\nonumber && + \limsup_{n\to\infty}\sum_{l = 1}^k\int_{B_{3R/2}\backslash B_R}dy\int_{B^c_{2R}\cap B_{4R}}\frac{|v^l_n(x) - v^l_n(y)|^2}{|x - y|^{N + 2s}}dx\\
\nonumber &\leq& o_R(1) + \limsup_{n\to\infty}\sum_{l = 1}^k\int_{B_{4R}\cap B^c_{3R/2}}dy\int_{B_{4R}\cap B^c_{3R/2}}\frac{|v^l_*(x) - v^l_*(y)|^2}{|x - y|^{N + 2s}}dx\\
\nonumber & =& o_R(1).
\end{eqnarray}
Similarly, we get
\begin{align*}
    \limsup_{n\to\infty}  |I^{(21)}_n| &\leq o_R(1)
\end{align*}
and
\begin{align*}
  \limsup_{n\to\infty}|I^{(31)}_n| + |I^{(32)}_n| + |I^{(33)}_n|\leq o_R(1).
\end{align*}
Therefore
\begin{equation}\label{impotrant}
\varepsilon^{2s - N}|I_n| \leq o_R(1).
\end{equation}
Letting $R\to\infty$, we  have
$$
\liminf_{n\to \infty}\frac{J_{\varepsilon_n}(u_n)}{{\varepsilon^N_n}}\geq\sum_{j = 1}^{k}\mathcal{C}(x^j_*).
$$
\end{proof}

\begin{remark}\label{re3.5}
The nonlocal effect of $(-\Delta)^s$ makes us have to find the estimate on the global $L^2$-norm of the rescaled function $v^l_n(y)$, see \eqref{eq3.7}, \eqref{eq3.8}. This is quite different to the local operator $-\Delta$. Actually, for local operator, we only need some local information on $u_{\varepsilon}$, see the proof of Proposition 3.3 in  \cite{28} for example.

For the global $L^2$ norm of the rescaled function $v^l_n(y)$, we know nothing but
$$
\int_{\mathbb{R}^N}V(\varepsilon_n y + x_{\varepsilon_n})(v^l_n(y))^2dy = \frac{1}{\varepsilon^N_n}\int_{\mathbb{R}^N}V(x)|u_n(x)|^2 < +\infty,
$$
so, it is naturally to test it against with potential term: $V(\varepsilon_n y + x_{\varepsilon_n})$. Then, it is crucial in using the fact
$$
\liminf_{|y|\to \infty}V(\varepsilon_n y + x^j_{n})|\varepsilon_n y + x^j_{n}|^{2s} > 0,
$$
which comes from our assumption $(V_2)$ and which implies that for all $|y|\geq 2R$
\begin{align*}
 &\quad V^l_n(y)|{\varepsilon_n} y + x^j_n|^{2s}|{\varepsilon_n} y + x^j_n|^{-2s}|y|^{2s}|y|^{-2s}
(|y| - R)^{N + 2s}\\
 &\geq \frac{R^{N}}{2^{2s}} V^l_n(y)|{\varepsilon_n} y + x^j_n|^{2s}|{\varepsilon_n} y + x^j_n|^{-2s}|y|^{2s}\\
 &\geq \left\{
         \begin{array}{ll}
           \frac{R^{N + 2s}}{2^{2s}}\inf_{\Lambda}V(x), & \text{if}\ {\varepsilon_n} y + x^j_n\in \Lambda  \\
           \frac{R^{N}}{2^{2s}}\Big(\frac{|y|}{\varepsilon_n|y| + |x^j_n|}\Big)^{2s} \inf_{x\in \Lambda^c}V(x)|x|^{2s}, & \text{if}\ {\varepsilon_n} y + x^j_n\in \Lambda^c.
         \end{array}
       \right.
\end{align*}

\end{remark}

    In the following, we let $U \subset \mathbb{R}^N$ be smooth bounded such that $\inf_U V > 0$ and $\bar{\Lambda}\subset U$. Combining Lemma \ref{le3.2}, and Proposition \ref{pr3.3}, we can obtain the concentration phenomena.

\begin{lemma}\label{le3.6}
 Let $\rho > 0$. There exists a family of points $\{x_{\varepsilon}\}\subset\Lambda$ such that
\begin{eqnarray*}
  &&\liminf_{\varepsilon\to 0}\|u_{\varepsilon}\|_{L^{\infty}(B_{\varepsilon\rho}(x_{\varepsilon}))}> 0,\\
  &&\lim\limits_{\varepsilon\to 0}V(x_{\varepsilon}) = \inf_{\Lambda}V,\\
  &&\lim\limits_{\varepsilon\to 0}dist(x_{\varepsilon},\mathbb{R}^N\backslash\Lambda) > 0,\\
  &&\lim\limits_{{R\to \infty}\atop{{\varepsilon\to 0}}}\|u_{\varepsilon}\|_{L^{\infty}(U\backslash B_{\varepsilon R}(x_{\varepsilon}))} =  0.
\end{eqnarray*}

\end{lemma}

\begin{proof} The fact that
$$
\limsup_{n\to\infty}\frac{1}{\varepsilon^N_n}\int_{\mathbb{R}^N}\varepsilon^{2s}_n|(-\Delta)^{s/2}u_n|^2 + V|u_n|^2 < \infty,
$$
follows from Lemma  \ref{le3.2} and the proof of Lemma \ref{le2.5}. Since
$u_{\varepsilon}$ is a critical point, by assumption $(P_2)$, for small $\varepsilon$,
\begin{align*}
  & \int_{\mathbb{R}^N}\varepsilon^{2s}|(-\Delta)^{s/2} u_{\varepsilon}|^2 + Vu^2_{\varepsilon}dx = \int_{\mathbb{R}^N}\mathfrak{g}_{\varepsilon}(u_{\varepsilon})u_{\varepsilon}\\
\leq &\tau\int_{\mathbb{R}^N}\varepsilon^{2s}|(-\Delta)^{s/2} u_{\varepsilon}|^2 + Vu^2_{\varepsilon}dx + \int_{\Lambda}|u_{\varepsilon}|^pdx.
\end{align*}
Then
\begin{align*}
  & \int_{\mathbb{R}^N}\varepsilon^{2s}|(-\Delta)^{s/2} u_{\varepsilon}|^2 + Vu^2_{\varepsilon} \leq C\||u_{\varepsilon}|^{p -2}\|_{L^{\infty}(\Lambda)}\int_{\mathbb{R}^N}\varepsilon^{2s}|(-\Delta)^{s/2} u_{\varepsilon}|^2 + Vu^2_{\varepsilon},
\end{align*}
where $C$ depending only on $s$, $N$ and $p$. From the regularity estimate in Appendix D in \cite{20}, $u_{\varepsilon}(x)$ achieves the  maximum on $\bar{\Lambda}$. Assume  that
$x_{\varepsilon}$ is a maximum point, we  have
$$
\liminf_{\varepsilon\to 0}\|u_{\varepsilon}\|_{L^{\infty}(B_{\varepsilon\rho}(x_{\varepsilon}))}\geq C' > 0.
$$
Take any subsequence $\{x_{\varepsilon_n}\}\subset \{x_{\varepsilon}\}$ satisfying  $\lim\limits_{n\to\infty}x_{\varepsilon_n} = x_{*}$. By the lower bound on the energy of Proposition \ref{pr3.3},
$$
\inf_{x\in\Lambda}\mathcal{C}(x)\geq\liminf_{n\to\infty}\frac{J_{\varepsilon_n}(u_{\varepsilon_n})}{\varepsilon^N_n}\geq \mathcal{C}(x_*) =  \lim\limits_{n\to\infty}\mathcal{C}(x_n).
$$
But the estimate of Lemma \ref{le3.2} tells us that
\begin{eqnarray*}
  \lim\limits_{n\to\infty}\mathcal{C}(x_n)\geq \inf_{x\in\Lambda}\mathcal{C}(x) \geq \liminf_{n\to\infty}\frac{J_{\varepsilon_n}(u_{\varepsilon_n})}{\varepsilon^N_n}.
\end{eqnarray*}
So
$$
\mathcal{C}(x_*)  = \inf_{x\in\Lambda}\mathcal{C}(x).
$$
Using Proposition \ref{pr3.1} again, we have
$$
V(x_*) = \inf_{\Lambda}V.
$$
Since $\inf_{\Lambda}V\leq\inf_{\partial\Lambda}V$, by the arbitrary of the sequence,
$$
\liminf_{\varepsilon\to 0} dist(x_{\varepsilon},\mathbb{R}^N\backslash \Lambda) > 0.
$$
Finally, by contradiction we assume that there exists a sequence of positive numbers $(\varepsilon_n)_{n\in \mathbb{N}}$ converging to $0$ and a sequence of points $z_n$ in $U$ such that
$$
\liminf_{n\to \infty}\|u_{\varepsilon_n}\|_{L^{\infty}(B_{\varepsilon_n\rho}(z_n))} > 0
$$
and
$$
\lim\limits_{n\to \infty}\frac{|z_n - x_n|}{\varepsilon_n} = \infty.
$$
Since $\overline{U}$ is compact, we can assume that the sequence $(z_n)$ converges to $z_*$ and  $x_n$ converges to $x_*$. Then by Proposition \ref{pr3.3}, $z_*,\ x_*\in \bar{\Lambda}$ and
$$
\liminf_{n\to \infty}{\frac{1}{\varepsilon^N_n}}J_{\varepsilon_n}(u_{\varepsilon_n}) \geq \mathcal{C}(x_*)  + \mathcal{C}(z_*) \geq 2\inf_{x\in\Lambda}\mathcal{C}(x) ,
$$
which is a contradiction to Lemma \ref{le3.2}.
\end{proof}

\vspace{1cm}
\section{Back to the original problem}

\noindent In this section we  prove that if we choose a suitable penalized function $P_{\varepsilon}$, then  for small $\varepsilon$, each solution $u_{\varepsilon}$ of the penalized problem $\mathcal{Q}_{\varepsilon}$ solves the original equation \eqref{eq1.1}. Our idea is to prove $u_{\varepsilon}^{p - 2}\leq P_{\varepsilon}$ on $\mathbb{R}^N\backslash \Lambda$ through comparison principle. As we can see in \eqref{eq4.12}, the comparison principle for nonlocal operator $(-\Delta)^s$ requires  global information of $u_{\varepsilon}$. By the regularity asserts in Appendix D in \cite{20}, we can assume that for small $\varepsilon$,
\begin{equation}\label{eq4.0}
\sup_{x\in \Lambda}|u_{\varepsilon}(x)|\leq C_{\infty}.
\end{equation}
We first linearize the penalized problem outside small balls.

\begin{proposition}\label{pr4.1}
(Linear equation outside small balls) For $\varepsilon > 0$ small enough and $\delta\in (0,1)$, there exist $R > 0$, $x_{\varepsilon}\in \Lambda$ such that
\begin{equation}\label{eql4.2}
  \left\{
    \begin{array}{ll}
      \varepsilon^{2s}(-\Delta)^{s} u_{\varepsilon} + (1 - \delta)Vu_{\varepsilon}\leq P_{\varepsilon}u_{\varepsilon}, & \text{in}\ \mathbb{R}^N\backslash  B_{R\varepsilon}(x_{\varepsilon}), \vspace{2mm}\\
      u_{\varepsilon}\leq C_{\infty} & \text{in}\ \Lambda.
    \end{array}
  \right.
\end{equation}
\end{proposition}

\begin{proof}
 That $u_{\varepsilon}\leq C_{\infty}$ in $\Lambda$ is from \eqref{eq4.0}. Since $p > 2$ and $\inf_{U} V(x) > 0$, by Lemma \ref{le3.6}, there exists $R > 0$ such that
$$
(u_{\varepsilon})^{p - 2}_{+}\leq \delta V\ \text{in}\ U\backslash B_{R\varepsilon }(x_{\varepsilon}).
$$
Obviously
$$
\mathfrak{g}_{\varepsilon}(u_{\varepsilon})
\leq
P_{\varepsilon}u_{\varepsilon}\ \text{in}\ \mathbb{R}^N\backslash U.
$$
Hence we conclude our result by inserting the previous pointwise bounds into the penalized equation $\mathcal{Q}_{\varepsilon}$.

\end{proof}

\subsection{Proof of Theorem \ref{th1.1}}
\vspace{0.5cm}

\noindent We now construct barrier functions for the linearize equation \eqref{eql4.2}, in order to show that solutions of penalized problem $(\mathcal{Q}_{\varepsilon})$ solve the original problem \eqref{eq1.1} for sufficiently small $\varepsilon > 0$.

The following embedding lemma is similar to Theorem 4 in \cite{30}, which is useful in constructing the penalized function.

\begin{lemma}\label{le4.3}
Let $K:\mathbb{R}^N\to \mathbb{R}^+$ and $V:\mathbb{R}^N\to \mathbb{R}^+$ be measurable functions. Set
$$
\mathcal{W}(x) = \frac{K(x)}{V(x)^{\frac{N}{2s} - \frac{q}{2}\Big(\frac{N}{2s} - 1\Big)}}.
$$

(i) If $\mathcal{W}\in L^{\infty}(\mathbb{R}^N)$ and $2 \leq q \leq 2^*_s$, then one has the continuous embedding
$$
H^s_{V,\varepsilon}(\mathbb{R}^N)\subset L^q(\mathbb{R}^N,K\mathcal{L}^{{N}}).
$$

(ii) If moreover $K \in L^{\infty}_{loc}(\mathbb{R}^N)$, $ q < 2^*_s$ and for every $\varepsilon > 0$,
$$
\mathcal{L}^{\mathbb{N}}(\{x\in \mathbb{R}^N|\mathcal{W}(x) > \varepsilon\}) < +\infty,
$$
then this embedding is compact.
\end{lemma}

For simplicity, in the sequel, we define
\begin{equation}\label{eq4.2}
  \theta_q : = \frac{N}{2s} - \frac{q}{2}\Big(\frac{N}{2s} - 1\Big).
\end{equation}

\begin{proof}
Proof of (i). Actually, $\theta_q 2 + (1 - \theta_q) 2^*_s = q$. For $u\in H^s_{V,\varepsilon}(\mathbb{R}^N)$, by H\"{o}lder inequality and Young's inequality
\begin{align*}
 \ds\int_{\mathbb{R}^N}|u(x)|^qK(x)dx&\leq C\ds\int_{\mathbb{R}^N}|u(x)|^{\theta_q 2 + \theta_q 2^*_s}V(x)^{\theta_q}dx\\
 &\leq \Big(\int_{\mathbb{R}^N}|u(x)|^
 2V(x)dx\Big)^{\theta_q}\Big(\int_{\mathbb{R}^N}|u(x)|^{2^*_s}dx\Big)^{1 -\theta_q}\\
 &\leq C\Big(\int_{\mathbb{R}^N}|u(x)|^
 2V(x)dx\Big)^{\theta_q}\Big(\|u\|_{H^s_{V,\varepsilon}(\mathbb{R}^N)}\Big)^{(1 -\theta_q)2^*_s}\\
 &\leq C \Big(\|u\|_{H^s_{V,\varepsilon}(\mathbb{R}^N)}\Big)^{\frac{q}{2}},
\end{align*}
which completes the proof of (i).

Proof of (ii). It is sufficient to show that for any $\epsilon > 0$, there exists a set $A\subset \mathbb{R}^N$ of finite measure such that for any $u\in H^s_{V,\varepsilon}(\mathbb{R}^N)$ with $\|u\|_{H^s_{V,\varepsilon}(\mathbb{R}^N)}\leq 1$,
$$
\int_{A^c}K(x)|u(x)|^qdx <\epsilon.
$$
Actually, letting  $A_{\delta} = \{x\in\mathbb{R}^N:\mathcal{W}(x)\geq \delta\}$, we  get
\begin{align*}
 \ds\int_{\mathbb{R}^N\backslash A_{\delta}}|u(x)|^qK(x)dx
 &\leq \delta\Big(\int_{\mathbb{R}^N\backslash A_{\delta}}|u(x)|^
 2V(x)dx\Big)^{\theta_q}\Big(\int_{\mathbb{R}^N}|u(x)|^{2^*_s}dx\Big)^{1 -\theta_q},\\
\end{align*}
then the conclusion follows.
\end{proof}

By the above lemma, we have the following observation.
\begin{corollary}\label{cr4.3}
Assume that
$$
\liminf_{|x|\to\infty}V(x)|x|^{2s} > 0
$$
and $K(x)$ is a nonnegative continuous function such that
$$
\limsup_{|x|\to\infty}K(x)|x|^{t} <+\infty
$$
with $t > 2s$. Then the embedding
$$
H^s_{V,\varepsilon}(\mathbb{R}^N)\subset L^2(\mathbb{R}^N,K\mathcal{L}^{{N}})
$$
is compact.
\end{corollary}

\begin{proof}
Obviously,
$$
\limsup_{|x|\to\infty}\frac{K(x)}{V(x)^{\theta_2}} \leq \limsup_{|x|\to\infty}\frac{C}{|x|^{t - 2s}} = 0.
$$
Hence we conclude our result  by Lemma~\ref{le4.3}.
\end{proof}

  The following proposition involves computation of $(-\Delta)^s$, which  is an essential step to construct a barrier function. Comparing our computation with that of  $-\Delta$ in Proposition 4.3 in \cite{28},  we can not compute $(-\Delta)^sf$ as  precisely as $-\Delta f$. In general, we can only estimate $(-\Delta)^sf$ (see \eqref{eqc1}, \eqref{eq4.5}, \eqref{eq4.6}), and the estimate is much harder than that of $-\Delta f$. A well-known fact that $w = \frac{1}{|x|^{N - 2s}}$ is a fundamental solution for $(-\Delta)^s$ (see \cite{20} and the reference therein), i.e.,
$$
(-\Delta)^s w = 0\ \ \ \text{in} \mathbb{R}^N\backslash\{0\},
$$
helps us a lot in the following construction.

In the following, we denote $d = \max\{|x|:\,\,x\in\partial{\Lambda}\}$.
\begin{proposition}\label{pr4.4}
Let  $V\in C(\mathbb{R}^N;[0,\infty))$ satisfy
$$
\liminf_{|x|\to\infty}V(x)|x|^{2s} > 0.
$$
Then there exists a positive $w_{\mu}\in H^s_{V,\varepsilon}(\mathbb{R}^N)$ such that
$$
\left\{
  \begin{array}{ll}
    \varepsilon^{2s}(-\Delta)^s w_{\mu}(x) + (1 - \delta)V w_{\mu}(x)\geq C\chi_{\Lambda} +\chi_{\mathbb{R}^N\backslash \Lambda} \frac{w_{\mu}(x)}{|x|^{2s}}, & in\ \mathbb{R}^N, \vspace{2mm}\\
    w_{\mu}(x) = \frac{1}{d}, & on\ \bar{\Lambda}.
  \end{array}
\right.
$$
\end{proposition}
\begin{proof}
We define $\Lambda_{\sigma} = \{x:dist(x,\Lambda) < \sigma\}$ with $\Lambda_{\sigma}\subset\subset U$. Choose $0 \leq \eta_{\sigma}(x)\leq 1$ be a smooth function to such that $\eta_{\sigma}\equiv 1$ on $\bar{\Lambda}$ and $\eta_{\sigma}\equiv 0$ on $\Lambda^c_{\sigma}$. For $0 < \mu < N - 2s$, we take $w_{\mu}(x) = \eta_{\sigma}(x)\frac{1}{d^{\mu}} + (1 - \eta_{\sigma}(x))\frac{1}{|x|^{\mu}}$. We then have $w_{\mu} = \frac{1}{d^{\mu}}$ on $\bar{\Lambda}$, $w_{\mu} > 0$ on $\overline{U}\backslash \Lambda$ and for  $x\in \mathbb{R}^N\backslash \overline{U}$,
\begin{equation}\label{eq4.1}
w_{\mu}(x) = \frac{1}{|x|^{\mu}}.
\end{equation}

We compute that for  $x\in \mathbb{R}^N\backslash U$,
\begin{align}\label{eqc1}
\nonumber &\quad(-\Delta)^sw_{\mu}(x)\\
\nonumber& = \int_{\bar{\Lambda}}\frac{|x|^{-\mu} - {d^{-\mu}}}{|x - y|^{N + 2s}}dy + \int_{\mathbb{R}^N\backslash{\bar{\Lambda}}}\frac{|x|^{-\mu} - w_{\mu}(y)}{|x - y|^{N + 2s}}dy\\
\nonumber                  & = \int_{\bar{\Lambda}}\frac{|x|^{-\mu} - {d^{-\mu}}}{|x - y|^{N + 2s}}dy + \int_{\mathbb{R}^N\backslash{\bar{\Lambda}}}\frac{|x|^{-\mu} - |y|^{-\mu}}{|x - y|^{N + 2s}}dy + \int_{\mathbb{R}^N\backslash{\bar{\Lambda}}}\frac{\eta_{\sigma}(y)|y|^{-\mu} - \eta_{\sigma}(y){d^{-\mu}}}{|x - y|^{N + 2s}}dy\\
                  & := \int_{\bar{\Lambda}}\frac{|x|^{-\mu} - {d^{-\mu}}}{|x - y|^{N + 2s}}dy + \int_{\mathbb{R}^N\backslash{\bar{\Lambda}}}\frac{|x|^{-\mu} - |y|^{-\mu}}{|x - y|^{N + 2s}}dy + I_{\sigma}\\
 \nonumber                 & \geq \frac{1}{|x|^{\mu + 2s}}\int_{\mathbb{R}^N}\frac{|z|^{\mu} - 1 }{|z|^{\mu}|z - \vec{e}_1|^{N + 2s}}dz + I_{\sigma}\\
 \nonumber                 & \geq \frac{w_{\mu}(x)}{|x|^{2s}}\Big(\int_{B_1(0)}\frac{|z|^{\mu} - 1 }{|z|^{\mu}|z - \vec{e}_1|^{N + 2s}}dz +\int_{(B_1(0))^c}\frac{|z|^{\mu} - 1 }{|z|^{\mu}|z - \vec{e}_1|^{N + 2s}}dz\Big) + I_{\sigma}\\
 \nonumber                 & = \frac{w_{\mu}(x)}{|x|^{2s}}\Big(\int_{(B_1(0))^c}\frac{1- |z|^{\mu}}{|z|^{N - 2s}|z - \vec{e}_1|^{N + 2s}}dz + \int_{(B_1(0))^c}\frac{|z|^{\mu} - 1 }{|z|^{\mu}|z - \vec{e}_1|^{N + 2s}}dz\Big) + I_{\sigma}\\
 \nonumber                &\geq \widetilde{C}\frac{w_{\mu}(x)}{|x|^{2s}} + I_{\sigma}.
\end{align}
Now we estimate $I_{\sigma}$. By the construction of $\eta_{\sigma}$,
\begin{align*}
I_{\sigma}& = \int_{\mathbb{R}^N\backslash{\bar{\Lambda}}}\frac{\eta_{\sigma}(y)|y|^{-\mu} - \eta_{\sigma}(y){d^{-\mu}}}{|x - y|^{N + 2s}}dy\\
          & \geq \int_{\Lambda_{\sigma}\backslash{\bar{\Lambda}}}\frac{|y|^{-\mu} - {d^{-\mu}}}{|x - y|^{N + 2s}}dy\\
          & = \frac{1}{|x|^{2s + \mu}}\int_{\frac{\Lambda_{\sigma}\backslash{\bar{\Lambda}}}{|x|}}\frac{|y|^{-\mu} -  d^{-\mu}|x|^{\mu}}{(|x/|x| - y|)^{N + 2s}}dy\\
          & = o(\sigma)\frac{1}{|x|^{\mu + 2s}}.
\end{align*}
For  $x\in \Lambda$,
\begin{align}\label{eq4.5}
(-\Delta)^sw_{\mu}(x)& = \int_{\mathbb{R}^N\backslash \Lambda}\frac{(1 - \eta_{\sigma}(y))(d^{-\mu} - |y|^{-\mu})}{|x - y|^{N + 2s}}dy \geq -C,
\end{align}
and for  $x\in U\backslash \Lambda$,
\begin{align}\label{eq4.6}
(-\Delta)^sw_{\mu}(x) = \frac{1}{2}\int_{\mathbb{R}^N}\frac{2w_{\mu}(x) - w_{\mu}(x + z) - w_{\mu}(x - z)}{|z|^{N + 2s}}dy
                  & \geq -C.
\end{align}
Hence, for $\varepsilon$ small enough, we have that
\begin{equation}\label{eq4.7}
\varepsilon^{2s}(-\Delta)^s w_{\mu}(x) + (1 - \delta)Vw_{\mu}(x)\geq \widetilde{C}\chi_{\Lambda} + \widetilde{C}\chi_{\mathbb{R}^N\backslash\Lambda}\frac{w_{\mu}(x)}{|x|^{2s}}.
\end{equation}
\end{proof}

\begin{remark}\label{re4.5}
We have used the fact that for every $u\in C^{1,1}(\mathbb{R}^N)$
\begin{align*}
  2(-\Delta)^s u(x) &= 2\int_{\mathbb{R}^N}\frac{u(x) - u(y)}{|x - y|^{N + 2s}}dy\\
                    &= \int_{\mathbb{R}^N}\frac{u(x) - u(x + z)}{|z|^{N + 2s}}dz + \int_{\mathbb{R}^N}\frac{u(x) - u(x + z)}{|z|^{N + 2s}}dz\\
                    &= \int_{\mathbb{R}^N}\frac{u(x) - u(x + z)}{|z|^{N + 2s}}dz + \int_{\mathbb{R}^N}\frac{u(x) - u(x - z)}{|z|^{N + 2s}}dz\\
                    &= \int_{\mathbb{R}^N}\frac{2u(x) - u(x + z) - u(x - z)}{|z|^{N + 2s}}dz
\end{align*}
in the estimate of \eqref{eq4.6}, by which we can eliminate the singularity caused by the usual expression $\frac{u(x) - u(y)}{|x - y|^{N + 2s}}$ when $x\to y$. However, for Laplacian, we have $-\Delta w_{\mu} = 0$ on $\Lambda$, and we can easily find  $|-\Delta w_{\mu}|\leq C$ on $U$, which is much easier than our estimates \eqref{eqc1}, \eqref{eq4.5} and \eqref{eq4.6}.
\end{remark}

\vspace{0.5cm}
\noindent Using the above computation of $(-\Delta)^s w_{\mu}$, we can construct a family of supersolutions to the linearize equation \eqref{eql4.2} in Proposition \ref{pr4.1} and construct the penalized function $P_{\varepsilon}$. As we mentioned before, the nonlocal effect of $(-\Delta)^s$ makes us have to construct the supersolutions globally, which is a great difference  from $-\Delta$ since for local operator $-\Delta$, it is easy to find a supersolution, see the proof of Proposition 4.4 in \cite{28}.

\begin{proposition}\label{pr4.5}
(Construction of barrier functions) Let  $V\in C(\mathbb{R}^N;[0,\infty))$ satisfy
$$
\liminf_{|x|\to\infty}V(x)|x|^{2s} > 0.
$$
If $\{x_{\varepsilon}\}\subset \Lambda$ is a family of points such that $\liminf_{\varepsilon\to 0}dist(x_{\varepsilon},\partial\Lambda) > 0$. Then for sufficiently small $\varepsilon > 0$, there exists $\overline{U}_{\varepsilon}\in H^s_{V,\varepsilon}(\mathbb{R}^N)\cap C^{1,1}(\mathbb{R}^N)$ and $P_{\varepsilon}$ satisfying the assumptions $(P_1)$ and $(P_2)$ in section 2  such that $\overline{{U}}_{\varepsilon} > 0$
satisfies
\begin{equation*}
  \left\{
    \begin{array}{ll}
      \varepsilon^{2s}(- \Delta)^s \overline{U}_{\varepsilon} + (1 - \delta)V\overline{U}_{\varepsilon}\geq P_{\varepsilon}\overline{U}_{\varepsilon}, & \text{in}\ \mathbb{R}^N\backslash B_{R\varepsilon}(x_{\varepsilon}),\vspace{2mm} \\
      \overline{U}_{\varepsilon}\geq C_{\infty}, & \text{in}\  B_{R\varepsilon}(x_{\varepsilon}).
    \end{array}
  \right.
\end{equation*}
Moreover, $\overline{{U}}^{p - 2}_{\varepsilon} < P_{\varepsilon}$ in $\mathbb{R}^N\backslash \Lambda$.
\end{proposition}

\begin{proof}
Let $r = 2\liminf_{\varepsilon\to 0}dist (x_{\varepsilon},\partial\Lambda)$. Define
$$
p_{\varepsilon}(x) = \left\{
                       \begin{array}{ll}
                         \overline{C}\Big(1 + \frac{\nu (r - |x - x_{\varepsilon}|)^{\beta}}{\varepsilon^{2s}}\Big), & \text{in}\ B_r(x_{\varepsilon}), \vspace{2mm}\\
                         \overline{C},& \text{on}\ \mathbb{R}^N\backslash B_r(x_{\varepsilon}),
                       \end{array}
                     \right.
$$
and
\begin{equation}\label{eq4.8}
\overline{U}_{\varepsilon}(x) = \varepsilon^{2s}p_{\varepsilon}(x)\bar{u}_{\varepsilon}(x) = \varepsilon^{2s}p_{\varepsilon}(x)w_{\mu}(x),
\end{equation}
where $\beta > (2s - 1)^+ + 1$ and the constant $\overline{C}$ will be determined later. Then for  $x\in \mathbb{R}^N\backslash B_{R\varepsilon}(x_{\varepsilon})$,
\begin{align}\label{est1}
\nonumber \varepsilon^{2s}(-\Delta)^s\overline{U}_{\varepsilon}(x)& = \varepsilon^{4s}p_{\varepsilon}(x)(-\Delta)^sw_{\mu}(x) + \varepsilon^{4s}\int_{\mathbb{R}^N}\frac{\bar{u}_{\varepsilon}(y)(p_{\varepsilon}(x) - p_{\varepsilon}(y))}{|x - y|^{N + 2s}}dy\\
&:= \varepsilon^{4s}p_{\varepsilon}(x)(-\Delta)^s w_{\mu}(x) + \varepsilon^{4s} I_{\varepsilon}(x).
\end{align}
\\
Now we estimate $I_{\varepsilon}$.
\\
\textsf{Case 1} $x\in B^c_r(x_{\varepsilon})\cap \Lambda$. Since $\beta > 2s$,
\begin{align}\label{est2}
\nonumber\varepsilon^{2s}I_{\varepsilon}(x) &= -\nu d^{-1}\int_{B_r(x_{\varepsilon})}\frac{(r - |y - x_{\varepsilon}|)^{\beta}}{|x - y|^{N + 2s}}dy\\
  &= -\nu d^{-1}\int_{B_r(0)}\frac{(r - |y|)^{\beta}}{|x - x_{\varepsilon} - y|^{N + 2s}}dy\\
\nonumber  &\geq -\nu d^{-1}\int_{B_r(0)}\frac{(r - |y|)^{\beta}}{(|x - x_{\varepsilon}| - |y|)^{N + 2s}}dy\\
\nonumber  &\geq -C_1\nu d^{-1},
\end{align}
where $C_1 = C_1(N,\beta) > 0$ is a constant.

\textsf{Case 2} $x\in \Lambda^c$.
\begin{align}\label{est3}
\nonumber\varepsilon^{2s}I_{\varepsilon}(x) &= -\nu d^{-1}\int_{B_r(x_{\varepsilon})}\frac{(r - |y - x_{\varepsilon}|)^{\beta}}{|x - y|^{N + 2s}}dy\\
  &= -\nu d^{-1}\int_{B_r(0)}\frac{(r - |y|)^{\beta}}{|x - x_{\varepsilon} - y|^{N + 2s}}dy\\
\nonumber  &= -\nu d^{-1}|x|^{-(N + 2s)}\int_{B_r(0)}\frac{(r - |y|)^{\beta}}{|(x - x_{\varepsilon} - y)/|x||^{N + 2s}}dy\\
\nonumber  &\geq -\nu d^{-1}|x|^{-(N + 2s)}\int_{B_r(0)}\frac{(r - |y|)^{\beta}}{[(|x - x_{\varepsilon}| - |y|)/|x|]^{N + 2s}}dy\\
\nonumber  &\geq -\nu C_2|x|^{-(N + 2s)},
\end{align}
where $C_2 = C_2(\beta) > 0$ is a constant.

\textsf{Case 3} $x\in B^c_{R{\varepsilon}}(x_{\varepsilon})\cap B_r(x_{\varepsilon})$. We have
\begin{align}\label{eest4}
I_{\varepsilon}(x) &\geq d^{-1}\int_{B_r(x_{\varepsilon})}\frac{p_{\varepsilon}(x) - p_{\varepsilon}(y)}{|x - y|^{N + 2s}}dy = d^{-1}\int_{B_r(0)}\frac{p_{\varepsilon}(x) - p_{\varepsilon}(y + x_{\varepsilon})}{|x - x_{\varepsilon} - y|^{N + 2s}}dy.
\end{align}

Then for all $x\in B^c_{R{\varepsilon}}(0)\cap B_r(0)$, since $\beta > 2s$,
\begin{align}
\varepsilon^{2s} I_{\varepsilon}(x + x_{\varepsilon}) & \geq d^{-1}\nu\int_{B_r(0)}\frac{-(r - |y|)^{\beta}}{(r - |y|)^{N + 2s}}dy = d^{-1}\nu C_4,
\end{align}
where $C_4 = C_4(\beta,N) > 0$ is a constant.

We note that since the above constants $C_1,\ C_3,\ C_3$ and $C_4$  are all independent of $\overline{C}$, for $\nu$ small enough, we have
\begin{align}\label{est5}
\nonumber &\varepsilon^{2s}(-\Delta)^s\overline{U}_{\varepsilon}(x) + (1 - \delta)V(x)\overline{U}_{\varepsilon}(x)\\
\nonumber=& \varepsilon^{4s}p_{\varepsilon}(x)(-\Delta)^sw_{\mu}(x) + \varepsilon^{2s}(1 - \delta)V(x)p_{\varepsilon}(x)\overline{u}_{\varepsilon}(x) + \varepsilon^{4s}I_{\varepsilon}(x)\vspace{2mm}\\
\geq &\left\{
                      \begin{array}{ll}
                         \varepsilon^{2s}(1 - \delta)\overline{C}\big(\inf_{{\Lambda}}V)d^{-1} - \varepsilon^{2s}d^{-1}\nu \widetilde{C} , & x\in B_r(x_{\varepsilon})\cap B^c_{R\varepsilon}(x_{\varepsilon}) \vspace{2mm}\\
                         \varepsilon^{2s}(1 - \delta)\big(\inf_{{\Lambda}}V)\overline{C}d^{-1} - \varepsilon^{2s}d^{-1}\nu \widetilde{C}, & x\in B^c_r(x_{\varepsilon})\cap \Lambda \vspace{2mm}\\
                         \overline{C}\big(\inf_{x\in\Lambda^c}V(x)|x|^{2s})(1 - \delta)\varepsilon^{2s}|x|^{-2s - \mu}  - \varepsilon^{2s}C_2\nu |x|^{-(N + 2s)}, & \Lambda^c
                       \end{array}
                     \right.\\
\nonumber\geq &\left\{
                      \begin{array}{ll}
                        0, & x\in \Lambda \cap B^c_{R\varepsilon}(x_{\varepsilon})\vspace{2mm}\\
                        C\varepsilon^{2s}|x|^{-2s - \mu}, & \Lambda^c.
                       \end{array}
                     \right.
\end{align}
For the $\nu$ above we can choose $\overline{C} = \overline{C}(\nu,\varepsilon,\beta)$ large enough  such that $\overline{U}_{\varepsilon}(x)\geq C_{\infty}$ in $B_{R\varepsilon}(x_{\varepsilon})$.

Since $p > 2 + \frac{2s}{N - 2s}$, we let $\mu$  close to $N - 2s$, and then choose $t\in (2s,\mu(p - 2))$ and $\sigma\in (0,2s(p - 2))$ to define
\begin{align}\label{eq4.9}
  P_{\varepsilon}(x) = \frac{\varepsilon^{\sigma}}{|x|^{t}}\chi_{\mathbb{R}^N\backslash\Lambda}.
\end{align}
Obviously, for $\varepsilon$ small enough,
\begin{equation}\label{eq4.10}
\overline{U}^{p - 2}_{\varepsilon}(x) = \overline{C}^{p - 2}\frac{\varepsilon^{2s(p - 2)}}{|x|^{\mu(p - 2)}}\leq \frac{\varepsilon^{\sigma}}{|x|^{t}} = P_{\varepsilon}(x)\ \ \forall x\in \Lambda^c
\end{equation}
and
\begin{align}\label{eq4.11}
\left\{
  \begin{array}{ll}
    \varepsilon^{2s}(-\Delta)^s\overline{U}_{\varepsilon}(x) + (1 - \delta)V(x)\overline{U}_{\varepsilon}(x)\geq P_{\varepsilon}(x)\overline{U}_{\varepsilon}(x), & x\in B^c_{R\varepsilon}(x_{\varepsilon}),\vspace{2mm} \\
    \overline{U}_{\varepsilon}(x) \geq C_{\infty}, & x\in B_{R\varepsilon}(x_{\varepsilon}).
  \end{array}
\right.
\end{align}
Finally, by Corollary \ref{cr4.3}, it is easy to check  that $(P_1)$ is satisfied since $t > 2s$ and $(P_2)$ is satisfied for small $\varepsilon$ since $\sigma > 0$.

\end{proof}

\begin{remark}\label{re4.7}
Our construction \eqref{eq4.9} requires $p > 2 + \frac{2s}{N - 2s}$, which combines with  Remark \ref{re3.5} and Corollary \ref{cr4.3} implies that  the assumption $(V_2)$ on $V$ is necessary.
\end{remark}

\noindent Now  we complete the proof of  Theorem \ref{th1.1}.
\\
\textsf{Proof of Theorem} \ref{th1.1}. Putting $v_{\varepsilon} = u_{\varepsilon} - \overline{U}_{\varepsilon}$. Propositions \ref{pr4.1} and \ref{pr4.4} imply that
\begin{equation}\label{eq4.12}
\left\{
  \begin{array}{ll}
    \varepsilon^{2s}(-\Delta)^{s}v_{\varepsilon} + ((1 - \delta)V - P_{\varepsilon})v_{\varepsilon}\leq 0, & \text{in}\ \mathbb{R}^N\backslash B_{R{\varepsilon}}(x_{\varepsilon})\vspace{2mm}\\
    v_{\varepsilon}\leq 0, & \text{on}\ B_{R{\varepsilon}}(x_{\varepsilon}).
  \end{array}
\right.
\end{equation}
Then arguing by contradiction we can verify  that $v_{\varepsilon}\leq 0$ in $\mathbb{R}^N$. Hence $u^{p - 2}_{\varepsilon}\leq \overline{U}^{p - 2}_{\varepsilon}$ in $B^c_{R\varepsilon}(x_{\varepsilon})$,
 and by \eqref{eq4.10}, $u_{\varepsilon}^{p - 2} < P_{\varepsilon}$ on $\mathbb{R}^N\backslash\Lambda$.

 As a result,  $u_{\varepsilon}$ is indeed the solution of the original problem \eqref{eq1.1}.

\end{document}